\documentclass[12pt, reqno]{amsart}
\usepackage{latexsym,amsmath,amsopn,amssymb,amsthm,amsfonts}
\usepackage[T2A]{fontenc}
\usepackage[cp1251]{inputenc}
\usepackage{graphicx}

\newtheorem{theorem}{Theorem}

\newtheorem{proposition}{Proposition}

\newtheorem{remark}{Remark}

\newtheorem{corollary}{Corollary}

\DeclareMathOperator{\ad}{ad}

\DeclareMathOperator{\Ad}{Ad}

\DeclareMathOperator{\A}{A}

\DeclareMathOperator{\SO}{SO}

\DeclareMathOperator{\Sol}{Sol}

\DeclareMathOperator{\SL}{SL}

\DeclareMathOperator{\S^1}{S^1}

\begin{document}

\leftline {УДК 514.82 + 512.81 + 517.977+517.838+519.46}
\vspace{2mm}

\title[The G\"odel Universe and the Iwasawa decomposition]{The G\"odel Universe as a Lie group with left-invariant Lorentz metric and\newline the Iwasawa decomposition} 
\author{V.~N.~Berestovskii}
\thanks{The work was carried out within the framework of the State Contract to the IM SB RAS, project FWNF-2022-0006.}
\address{Sobolev Institute of Mathematics of the SB RAS,\newline
4 Acad. Koptyug Ave., Novosibirsk 630090, Russia}
\email{vberestov@inbox.ru}
\begin{abstract}
We discuss models of the G\"odel Universe as Lie groups with left-invariant Lorentz metric for two simply connected four-dimensional Lie groups, the Iwasawa decomposition for semisimple Lie groups, and left-invariant Lorentz metric on $\SL(2,\mathbb{R})$, following K.-H.~Neeb. Also we show that the isometry between two non-isomorphic sub-Riemannian Lie group, constructed by
A.~Agrachev and D.~Barilari, is induced by some Iwasawa decomposition of
$\SL(2,\mathbb{R})$.
\vspace{2mm}

2020 Mathematical Subject Classification 83C20, 53C17, 53C50, 49J15, 22E30   
	
\vspace{2mm}
\end{abstract}    	
\maketitle
                                                                         
\rightline {Dedicated to the memory of Vitaly Roman'kov}
\vspace{2mm}
                                                                         
\section{Introduction}

Kurt G\"odel in paper \cite{God1949} of 1949 introduced the Lorentz metric  (\ref{lig}) of the signature $(+,-,-,-)$ on the space $\mathbb{R}^4.$ The
G\"odel Universe (space-time) $S$ is a solution of the General Relativity Theory (the Einstein gravitation equations). 

In paper \cite{Ber2024}, the author found timelike and isotropic geodesics of the G\"odel Universe, considered as the Lie group $G=(\mathbb{R},+)\times \A^+(\mathbb{R})\times (\mathbb{R},+)$  with left invariant Lorentz metric.  Here $\A^+(\mathbb{R})$ is connected Lie group of affine transformations on $\mathbb{R}.$  

The above left-invariant Lorentz metric on $G$ and behaviour of its geodesics are defined essentially by the corresponding induced left-invariant Lorentz metric $ds^2_0$ on the subgroup $G_3=(\mathbb{R},+)\times \A^+(\mathbb{R}).$ 

Professor Karl--Hermann Neeb wrote to the author that it is possible to realize the G\"odel Universe otherwise. He sent an electronic version of his joint with Joachim Hilgert book \cite{HilgNeeb1993}, where in section 2.7 ''G\"odel's cosmological  model and universal covering of $\SL(2,\mathbb{R})$'' is suggested a left-invariant Lorentz metric on $\widetilde{\SL(2,\mathbb{R})}$ {\it which is isometric to $(G_3,ds^2_0)$} as stated there.  

In this connection, it is useful to mention the paper \cite{AgrBar2012} by A.~Agrachev and D.~Barilari, where the autors obtained a full classification of left-invariant sub-Riemannian metrics on three-dimensional Lie groups and ''explicitly find a sub-Riemannian isometry between  nonisomorphic Lie groups $\SL(2,\mathbb{R})$ and 
$\SO(2)\times A^+(\mathbb{R})$'' \cite{AgrBar2012}. 

The existence of such isometry was indicated ealier
in \cite{FalbGor1996} by Falbel and Gorodski.

In a message to the author, Professor Neeb explains the mentioned two isometries by a diffeomormism of Lie groups $\SL(2,\mathbb{R})$ and $\SO(2)\times A^+(\mathbb{R})$ by means of the Iwasawa decomposition for $\SL(2,\mathbb{R})$; let us cite Theorems 6.5.1 and 9.1.3 from \cite{Helg2001}.

In p. 10.6.4 (i) from \cite{Helg2001} are indicated the following isomorphisms of Lie algebras: 
$$\mathfrak{sl}(2,\mathbb{R})\cong \mathfrak{su}(1,1)\cong \mathfrak{so}(2,1)\cong \mathfrak{sp}(1,\mathbb{R}).$$
Consequently, simply connected Lie groups with these Lie algebras are isomorphic.

In Theorem \ref{main} we prove some properties of special left-invariant Lorentz metrics on  Lie groups which support the mentioned statement on the isometry of two left-invariant Lorentz metrics from \cite{HilgNeeb1993}. Also 
we show in Proposition \ref{PsiAB} that the isometry between two non-isomorphic sub-Riemannian Lie group, constructed by A.~Agrachev and D.~Barilari, is induced by some Iwasawa decomposition of 
$\SL(2,\mathbb{R})$. 

The author expresses his gratitudes to Professor Neeb for fruitful discussions and anonymous referee for useful remarks.  

\section{The G\"odel Universe as a Lie group with left-invariant Lorentz metric}
\label{subl}

G\"odel introduced in \cite{God1949} his space-time $S$ as $\mathbb{R}^4$ with the linear element
\begin{equation}
\label{lig}
ds^2=a^2\left(dx_0^2+2e^{x_1}dx_0dx_2+ \frac{e^{2x_1}}{2}dx_2^2-dx_1^2-dx_3^2\right),\quad  a>0. Ыд
\end{equation}
G\"odel noticed in \cite{God1949} that it is possible to rewrite this quadratic form in view of   
\begin{equation}
\label{lig1}
ds^2=a^2\left[\left(dx_0+e^{x_1}dx_2\right)^2-dx^2_1-
\frac{e^{2x_1}}{2}dx^2_2-dx^2_3\right],	  
\end{equation}
which shows obvious that its signature is equal everywhere to $(+,-,-,-).$

We shall assume that $a=1.$ 

G\"odel noticed in \cite{God1949} that on $(S,ds^2)$ acts simply transitively a four-dimensional isometry Lie group. It is easy to see that such action could be written as 
\begin{equation}
\label{gr}
x_0=x_0'+a,\quad x_1=x_1'+b,\quad x_2=x_2'e^{-b}+c,\quad x_3=x_3'+d
\end{equation}
with arbitrary $a,b,c,d\in\mathbb{R}.$ This implies that corresponding Lie group $G$ is the simplest simply connected noncommutative four-dimensional Lie group of the view 
\begin{equation}
\label{g1}	
G\cong [(\mathbb{R},+)\times G_{2}]\times (\mathbb{R},+),
\end{equation} 
where $G_2$ is unique up to isomorphism, necessary isomorphic to $\mathbb{R}^2$, two-dimensional noncommutative Lie group. 
The Lie group $G_2$ is isomorphic to the Lie group $A^+(\mathbb{R})$ of preserving orientation affine transformations of the real direct line  $(\mathbb{R},+).$  

In case under consideration, identifying the quad $(x'_0,x_1',x_2',x'_3)$ with 
the vector $(x_2', x_1', x_0',x_3',1)^T,$ where $^T$ is the sign of transposition,
the action of the group $G$ on $\mathbb{R}^4$ by formula $(\ref{gr})$ has the view $(x_2, x_1, x_0, x_3,1)^T=A(x_2', x_1', x_0',x_3',1)^T,$ where 
\begin{equation}
\label{A}
A= \left(\begin{array}{ccccc}
e^{-b} & 0 & 0 & 0 & c \\
0 & 1 & 0 & 0 & b \\ 
0 & 0 & 1 & 0 & a\\
0 & 0 & 0  & 1 & d \\
0 & 0 & 0 &  0 & 1  \end{array}\right).
\end{equation}
Under this the equality 
\begin{equation}
\label{B}
\left(\begin{array}{ccccc}
e^{-x_1} & 0 & 0 & 0 & x_2 \\
0 & 1 &  0 & 0 & x_1 \\ 
0 & 0 & 1 & 0 & x_0\\
0 & 0 & 0 & 1 & x_3\\
0 & 0 & 0 & 0  & 1 \end{array}\right)(0,0,0,0,1)^T=(x_2,x_1,x_0, x_3,1)^T
\end{equation}
sets the bijection of the group $G$ onto $\mathbb{R}^4$ and the unit of $G$ corresponds to the zero-vector $(0,0,0,0)\in\mathbb{R}^4.$ On base of this, (\ref{g1}) and (\ref{lig}), we can identify $(S,ds^2)$ with the Lie group $G$ equipped with left-invariant Lorentz metric. Let 
$$e_0=\frac{\partial} {\partial x_0}(0),e_1=\frac{\partial}{\partial x_1}(0),e_2=\frac{\partial}{\partial x_2}(0),e_3=\frac{\partial}{\partial x_3}(0))$$ 
be the basis of the Lie algebra $\mathfrak{g}$ of the Lie group $G$ at the unit of $G$, corresponding to coordinates $(x_0,x_1,x_2,x_3)$. Then, according to what has been said and (\ref{lig}), the components of the linear element $ds^2$ with respect to this basis are equal to 
\begin{equation}
\label{g} 
g_{00}=1, g_{02}=g_{20}=1, g_{22}=\frac{1}{2}, g_{11}=-1, g_{33}=-1, g_{ij}=g_{ji}=0, i\neq j, j\neq 2.
\end{equation}

According to (\ref{B}), the Lie subgroup $G_3:=(\mathbb{R},+)\times G_{2}$ can be identified with the matrix Lie group   
\begin{equation}
\label{C}
\left(\begin{array}{cccc}
e^{-x_1} & 0 & 0 &  x_2 \\
0 & 1 &  0 & x_1 \\ 
0 & 0 & 1 & x_0\\
0 & 0 & 0 &  1 \end{array}\right),\quad (x_0,x_1,x_2)\in \mathbb{R}^3.
\end{equation}

It is obvious that $(S,ds^2)=(S_0,ds_0^2)\times (S_1,ds_1^2),$ where $S_0=\mathbb{R}^3,$
$S_1=\mathbb{R},$
\begin{equation}
\label{lin}	
ds_0^2=dx_0^2+2e^{x_1}dx_0dx_2+ \frac{e^{2x_1}}{2}dx_2^2-dx_1^2,\quad ds_{1}^2=-dx_3^2.	
\end{equation}
Also it is clear that we can consider $(S_0,ds_0^2)$ as the matrix Lie group (\ref{C}) with left-invariant Lorentz metric, which according to (\ref{g}) has components 
\begin{equation}
\label{g0} 
g_{00}=1, g_{02}=g_{20}=1, g_{22}=\frac{1}{2}, g_{11}=-1, g_{ij}=g_{ji}=0, i\neq j, j\neq 2;
\end{equation}
with respect to the basis $e_0,e_1,e_2$ of the Lie algebra $\mathfrak{g}_3$
of the matrix Lie group (\ref{C}).

In consequence of (\ref{B}), for the Lie algebra $\mathfrak{g}_3$ of the Lie group $G_3,$ 
\begin{equation}
\label{e}	
e_0 = \left(\begin{array}{cccc}
0 & 0 & 0 & 0 \\
0 & 0 & 0 & 0 \\
0 & 0 & 0 & 1\\ 
0 & 0 & 0 & 0\end{array}\right),\quad	
e_1 = \left(\begin{array}{cccc}
-1 & 0 & 0 & 0 \\
0 & 0 & 0 & 1 \\
0 & 0 & 0 & 0\\ 
0 & 0 & 0 & 0\end{array}\right),\quad 
e_2 = \left(\begin{array}{cccc}
0 & 0 & 0 & 1 \\
0 & 0 & 0 & 0\\
0 & 0 & 0 & 0 \\ 
0 & 0 & 0 & 0\end{array}\right).	 
\end{equation}
Then in the Lie algebra $\mathfrak{g}_3$,
\begin{equation}
\label{sc}
[e_1,e_2]=e_1e_2-e_2e_1=-e_2, \quad [e_0,e_1]=[e_0,e_2]=0.	 
\end{equation}

\section{The Iwasawa decompositions of Lie algebras and Lie groups}

Let $\mathfrak{g}$ be a semisimple real Lie algebra, $\sigma$ be some Cartan involution of $\mathfrak{g},$ and $\mathfrak{g}=\mathfrak{k}\oplus \mathfrak{p}$ be the corresponding Cartan decomposition ($\mathfrak{k}$ is the Lie subalgebra of $\mathfrak{g}$, consisting of fixed points relative to $\sigma$). Let us denote by $\mathfrak{a}$ a maximal commutative subspace in
$\mathfrak{p}.$ Then there is the following {\it Iwasawa decomposition of Lie algebra} $\mathfrak{g}$.

\begin{theorem}
\label{GG} (4.7.2) in \cite{GotoGr1978}.
Let $\mathfrak{g}$ be a semisimple real Lie algebra. 

Then there exists a direct sum of vector subspaces in $\mathfrak{g}$
\begin{equation}
\label{Iwalg}	 
\mathfrak{g}=\mathfrak{k}\oplus \mathfrak{a}\oplus\mathfrak{n},
\end{equation}
where $\mathfrak{n}$ is a nilpotent subalgebra in $\mathfrak{g}$
such that the endomorphism $\ad X$ is nilpotent for every $X\in\mathfrak{n},$
and $\mathfrak{a}\oplus\mathfrak{n}$ is a solvable subalgebra in $\mathfrak{g}.$
\end{theorem}

As an example, the authors of \cite{GotoGr1978} give the decomposition
(\ref{Iwalg}) for $\mathfrak{g}=\mathfrak{sl}(n,\mathbb{R}).$ In this
case $\mathfrak{k}$ is the Lie subalgebra of skew-symmetric matrices,
$\mathfrak{a}$ is the Lie subalgebra of diagonal matrices with zero trace, and 
$\mathfrak{n}$ is the Lie subalgebra of strictly upper triangular matrices. 
In particular, for $\mathfrak{g}=\mathfrak{sl}(2,\mathbb{R})$ we have
\begin{equation}
\label{sl2}
\mathfrak{k}=
\left\{\left(\begin{array}{cc}
0 & t \\
-t & 0\end{array}\right)\right\},\quad \mathfrak{a}=
\left\{\left(\begin{array}{cc}
t & 0 \\
0 & -t\end{array}\right)\right\},\quad \mathfrak{n}=
\left\{\left(\begin{array}{cc}
0 & t \\
0 & 0\end{array}\right)\right\},\quad t\in\mathbb{R},
\end{equation}
with natural basis
\begin{equation}
\label{basel2}
f_0=
\left(\begin{array}{cc}
0 & 1 \\
-1 & 0\end{array}\right),\quad f_1=
\left(\begin{array}{cc}
1 & 0 \\
0 & -1\end{array}\right),\quad f_2=
\left(\begin{array}{cc}
0 & 1 \\
0 & 0\end{array}\right)
\end{equation}
and Lie brackets for this basis
\begin{equation}
\label{brack}
[f_0,f_1]=2f_0-4f_2,\quad [f_0,f_2]=f_1,\quad [f_1,f_2]=2f_2.	
\end{equation}

Let $K=\exp(\mathfrak{k}),$	$A=\exp(\mathfrak{a}),$ $N=\exp(\mathfrak{n})$
be Lie subgroups of the semisimple Lie group $G,$ corresponding to the decomposition (\ref{Iwalg}). 

\begin{theorem}
\label{Iwgr} (Theorem 9.1.3 in \cite{Helg2001})
Let $G$ be a connected semisimple real Lie group. Then $G=KAN$ and the mapping
\begin{equation}
\label{map}
(k,a,n)\rightarrow kan
\end{equation}
is the diffeomorphism of manifold $K\times A\times N$ onto the Lie group $G.$
\end{theorem}

\begin{corollary}
\label{diffeo}
The Lie group $G$ is diffeomorphic to Lie groups $K\times AN$ and 
$K\times A\times N.$ 
\end{corollary}

Theorem \ref{GG} and Theorem 6.1.1 from \cite{Helg2001} imply the following

\begin{proposition} 
\label{subgr}
The sets $K,$ $A,$ $N,$ and $AN$ are connected closed Lie subgroups of
the Lie group $G,$ where $\Ad_G(K)$ is compact, $A$ is commutative, $N$ is nilpotent, and $AN$ is solvable. The subgroup $K$ contains the center $Z$ of the Lie group $G.$ In addition, $K$ is compact if and only if the center $Z$ of $G$ is finite; in this case $K$ is a maximal compact subgroup of the Lie group $G.$
\end{proposition}

The statements above imply the following	

\begin{corollary} 
\label{subgrsl}
If $G=\SL(n,\mathbb{R}),$ then $K=\SO(n),$ $A$ is the group of all real diagonal $(n\times n)$-matrices with unit determinant, $N$ could be considered as the group of all real upper triangular $(n\times n)$-matrices with units on the main diagonal, and $\Sol(n):=AN$ as the group of all real upper triangular 
$(n\times n)$-matrices with unit determinant.   
\end{corollary} 

\begin{corollary}
\label{diffsl}	
The Lie group $\SL(n,\mathbb{R})$ is diffeomorphic to Lie groups $\SO(n)\times AN$ and $\SO(n)\times A\times N.$ As a consequence, $\SL(2,\mathbb{R})$ is diffeomorphic to the commutative Lie group $\SO(2)\times A\times N.$  	
\end{corollary}

\section{$(S_0,ds^2_0)$ as $(\mathbb{R},+)\times\Sol(2)$ with left-invariant Lorentz metric}
\label{prep}

This is a preparatory section. 

\begin{proposition}
\label{isomorphism}
There exist an isomorphism of the Lie group $G_3$ onto the Lie 
group $(\mathbb{R},+)\times\Sol(2)$ and corresponding realization of $(S_0,ds^2_0)$ as the Lie group $(\mathbb{R},+)\times\Sol(2)$ with left-invariant Lorentz metric. 
\end{proposition}

\begin{proof}
Comparing (\ref{sc}) and (\ref{brack}), we see that the linear map $\varphi: \mathfrak{a}\oplus\mathfrak{n}\rightarrow \mathfrak{g}_2$ such that
\begin{equation}
\label{isom}
\varphi\left(\frac{-f_1}{2}\right)=e_1,\quad \varphi(f_2)=e_2
\end{equation}
is an isomorphism of Lie algebras. Let $\Sol(2)$ be the Lie group with the Lie algebra 
$\mathfrak{a}\oplus\mathfrak{n}$ for (\ref{sl2}). Then (\ref{isom}) defines isomorphism of Lie groups $\psi:\Sol(2)\rightarrow G_2:$   
\begin{equation}
\label{isomgr}
\psi\left(\left(\begin{array}{cc}
e^{-s/2} & 0 \\
0 & e^{s/2}\end{array}\right)
\left(\begin{array}{cc}
1 & r \\
0 & 1\end{array}\right)\right)=\left(\begin{array}{cccc}
e^{-s} & 0 & 0 & e^{-s}r \\
0 & 1 & 0 & s\\
0 & 0 & 1 & 0\\
0 & 0 & 0 & 1\end{array}\right),
\end{equation}
\begin{equation}
\label{isomgr1}
\psi\left(\left(\begin{array}{cc}
1 & r \\
0 & 1\end{array}\right)
\left(\begin{array}{cc}
e^{-s/2} & 0 \\
0 & e^{s/2}\end{array}\right)\right)=\left(\begin{array}{cccc}
e^{-s} & 0 & 0 & r \\
0 & 1 & 0 & s\\
0 & 0 & 1 & 0\\
0 & 0 & 0 & 1\end{array}\right).
\end{equation}

We can consider $(\mathbb{R},+)$ as a Lie algebra and as a Lie group.  Then the mappings
\begin{equation}
\label{isoepi}
t\in (\mathbb{R},+)\rightarrow 
\left(\begin{array}{cc}
0 & t \\
-t & 0\end{array}\right), \quad t\in (\mathbb{R},+)\rightarrow 
\left(\begin{array}{cc}
\cos t & \sin t \\
-\sin t & \cos t\end{array}\right)
\end{equation}
are correspondingly the isomorphism of Lie algebras and respective universal covering epimorpism of Lie groups. 
Then there exists unique
isomorphism $\psi$ of the Lie group $(\mathbb{R},+)\times\Sol(2)$ onto the Lie group $G_3,$ with properties (\ref{isomgr}), (\ref{isomgr1}), and $\psi(t)=t$
for $t\in (\mathbb{R},+).$ It follows from previous considerations that we shall realize $(S_0,ds^2_0)$ as the Lie group $(\mathbb{R},+)\times\Sol(2)$ with left-invariant Lorentz metric if components of this metric in the basis 
$\{f_0, -f_1/2, f_2\}$ of its Lie algebra will be as in (\ref{g0}). 

The corresponding orthonormal basis is
\begin{equation}
\label{base}
X=f_0,\quad Y'=-f_1/2,\quad Z'=\sqrt{2}(f_0-f_2),
\end{equation}
which we can change by
\begin{equation}
\label{nbase}
X=f_0,\quad Y=f_1/2,\quad Z=\sqrt{2}(f_2-f_0).
\end{equation}
Then 
\begin{equation}
\label{YZ}
[Y,Z]=[Y',Z']=\sqrt{2}f_2=Z+\sqrt{2}X.
\end{equation}
\end{proof}

\begin{remark}
The basis (\ref{nbase}) has a form
\begin{equation}
\label{nXYZ}  
X=\left(\begin{array}{cc}
0 & 1 \\
-1 & 0\end{array}\right),
Y=\frac{1}{2}\left(\begin{array}{cc}
1 & 0 \\
0 & -1\end{array}\right),
Z=\sqrt{2}\left(\begin{array}{cc}
0 & 0 \\
1 & 0\end{array}\right).
\end{equation}
The group $\Sol(2)$ is isomorphic to the Lie group of real lower triangular 
$(2\times 2)-$matrices with unit determinant.
\end{remark}

\section{Left-invariant Lorentz metrics on $\SO(2)\times \Sol(2)$ and 
$\SL(2,\mathbb{R})$}

In \cite{HilgNeeb1993}, the authors consider the Lie group $\SL(2,\mathbb{R})$
with left-invariant Lorentz metric and orthonormal basis with the signature $(-,+,+)$ of the form
\begin{equation}
\label{XYZ}  
X=\frac{1}{\sqrt{2}}
\left(\begin{array}{cc}
0 & 1 \\
-1 & 0\end{array}\right),
Y=\left(\begin{array}{cc}
1 & 0 \\
0 & -1\end{array}\right),
Z=\left(\begin{array}{cc}
0 & 1 \\
1 & 0\end{array}\right)\in \mathfrak{sl}(2,\mathbb{R}).
\end{equation}
We shall consider this basis as orthonormal with the signature $(+,-,-).$

\begin{theorem}
\label{main}	
1) For the Lorentz metric on $\SL(2,\mathbb{R})$ from \cite{HilgNeeb1993} with
orthonormal basis (\ref{XYZ}) and for corresponding by the Iwasawa diffeomorphism basis on $\SO(2)\times \Sol(2)$, the curvatures $k(X,Y)=k(X,Z)=k(Y,Z)=-2,$ while  for the orthonormal basis (\ref{nbase}) on $(\mathbb{R},+)\times \Sol(2)$, 
$k(X,Y)=k(X,Z)=k(Y,Z)=-\frac{1}{2}$. 

2) One-parameter subgroups in $\SL(2,\mathbb{R}),$
defined by $X,$ $Y,$ $Z,$ are geodesics.   
\end{theorem}
\begin{proof}
For any (pseudo-)Riemannian manifold $M$ with (pseudo-)metric tensor $(\cdot,\cdot),$ the Levi-Civita connection $\nabla$, and smooth vector fields $X,$ $Y,$ $Z$ is valid the following equation (3.5.(7)) from \cite{GrKlMe1971}:
\begin{equation}
\label{nabla}
(\nabla_XY,Z)=\frac{1}{2}[X(Y,Z)+Y(Z,X)-Z(X,Y)+(Z,[X,Y])+(Y,[Z,X])-(X,[Y,Z])].	
\end{equation}
As a consequence, if $(M,(\cdot,\cdot))$ is a Lie group $G$ with left-invariant
(pseudo-)metric $(\cdot,\cdot)$ of the signature $(+,-,-)$ and $X,$ $Y,$ $Z$ are left-invariant, then 
\begin{equation}
\label{nabla1}
(\nabla_XY,Z)=\frac{1}{2}[(Z,[X,Y])+(Y,[Z,X])-(X,[Y,Z])].	
\end{equation}

It follows from (\ref{XYZ}) that
\begin{equation}
\label{brXYZ}
[X,Y]=-\sqrt{2}Z,\quad [X,Z]=\sqrt{2}Y,\quad [Y,Z]=2\sqrt{2}X.	
\end{equation}
Let us apply (\ref{nabla1}) and (\ref{brXYZ}) in the further computations.
$$(\nabla_XY,Z)=0,\quad (\nabla_XY,X)=(\nabla_XY,Y)=0,\quad \nabla_XY=0,$$
$$(\nabla_YY,X)=(\nabla_YY,Y)=(\nabla_YY,Z)=0,\quad \nabla_YY=0,$$
$$(\nabla_ZY,X)=-\sqrt{2},\quad (\nabla_ZY,Y)=(\nabla_ZY,Z)=0,\quad 
\nabla_ZY=-\sqrt{2}X,$$
$$\nabla_YZ=\nabla_ZY+[Y,Z]=\sqrt{2}X,$$
\begin{equation}
\label{Ri}
R(X,Y)Y=\nabla_X\nabla_YY-\nabla_Y\nabla_XY-\nabla_{[X,Y]}Y=
\sqrt{2}\nabla_ZY=-2X.
\end{equation}
Then $k(X,Y)=(R(X,Y)Y,X)=-2.$ Analogously, we obtain 
$$\nabla_XZ=0,\quad \nabla_XX=0,\quad \nabla_ZZ=0,\quad k(X,Z)=(R(X,Z)Z,X)=-2,$$
$$R(Y,Z)Z=\nabla_Y\nabla_ZZ-\nabla_Z\nabla_YZ-\nabla_{[Y,Z]}Z=
-\sqrt{2}\nabla_ZX=-\sqrt{2}(\nabla_XZ+[Z,X])=2Y,$$
$$k(Y,Z)=(R(Y,Z)Z,Y)=-2.$$

2) $[Y,Z]=2Z+2\sqrt{2}X$ in $\mathfrak{so}(2)\oplus\mathfrak{sol}(2).$ 

We look for $Z$ as $\alpha f_2+\beta X$,  see (\ref{basel2}). It is easy to see that
$\alpha=2,$ $\beta=-\sqrt{2},$  
$$[Y,Z]=[Y, 2f_2-\sqrt{2} X]=2[f_1,f_2]=4f_2=2(Z+\sqrt{2}X)=2Z+2\sqrt{2}X.$$

$$(\nabla_XY,Z)=-\sqrt{2},\quad (\nabla_XY,X)=(\nabla_XY,Y)=0,\quad \nabla_XY=\sqrt{2}Z,$$
$$(\nabla_YY,X)=(\nabla_YY,Y)=(\nabla_YY,Z)=0,\quad \nabla_YY=0,$$
$$(\nabla_ZY,X)=-\sqrt{2},\quad (\nabla_ZY,Y)=0,\quad (\nabla_ZY,Z)=2,\quad 
\nabla_ZY=-\sqrt{2}X-2Z,$$
$$\nabla_YZ=\nabla_ZY+[Y,Z]=\sqrt{2}X,$$
$$R(X,Y)Y=\nabla_X\nabla_YY-\nabla_Y\nabla_XY-\nabla_{[X,Y]}Y=
-\sqrt{2}\nabla_YZ=-2X.$$
Then $k(X,Y)=(R(X,Y)Y,X)=-2.$ Analogously, we obtain 
$$\nabla_XZ=-\sqrt{2}Y,\nabla_ZX=\nabla_XZ=-\sqrt{2}Y, \quad \nabla_ZZ=2Y,\quad 
k(X,Z)=(R(X,Z)Z,X)=-2,$$
$$R(Y,Z)Z=\nabla_Y\nabla_ZZ-\nabla_Z\nabla_YZ-\nabla_{[Y,Z]}Z=$$
$$\nabla_Y(2Y)-\nabla_Z(\sqrt{2}X)-\nabla_{(2Z+2\sqrt2X)}Z=$$
$$2Y-2\nabla_ZZ-2\sqrt{2}\nabla_XZ=2Y-4Y+4Y=2Y,$$
$$k(Y,Z)=(R(Y,Z)Z,Y)=-2.$$

Let us compute analogous expressions for $(S_0,ds_0^2),$ using results of 
Section \ref{prep}. 

$$(\nabla_XY,Z)=-\frac{\sqrt{2}}{2},\quad (\nabla_XY,X)=(\nabla_XY,Y)=0,\quad \nabla_XY=\frac{\sqrt{2}}{2}Z,$$
$$(\nabla_YY,X)=(\nabla_YY,Y)=(\nabla_YY,Z)=0,\quad \nabla_YY=0,$$
$$(\nabla_ZY,X)=-\frac{\sqrt{2}}{2},\quad (\nabla_ZY,Y)=0,\quad (\nabla_ZY,Z)=1,\quad 
\nabla_ZY=-\frac{\sqrt{2}}{2}X-Z,$$
$$\nabla_YZ=\nabla_ZY+[Y,Z]=\frac{\sqrt{2}}{2}X,$$
$$R(X,Y)Y=\nabla_X\nabla_YY-\nabla_Y\nabla_XY-\nabla_{[X,Y]}Y=
-\frac{\sqrt{2}}{2}\nabla_YZ=-\frac{1}{2}X.$$
Then $k(X,Y)=(R(X,Y)Y,X)=-\frac{1}{2}.$ Analogously, we obtain 
$$\nabla_XZ=-\frac{\sqrt{2}}{2}Y,\nabla_ZX=\nabla_XZ=-\frac{\sqrt{2}}{2}Y, \quad \nabla_ZZ=Y,\quad k(X,Z)=(R(X,Z)Z,X)=-\frac{1}{2},$$
$$R(Y,Z)Z=\nabla_Y\nabla_ZZ-\nabla_Z\nabla_YZ-\nabla_{[Y,Z]}Z=$$
$$\nabla_Y(Y)-\nabla_Z\left(\frac{\sqrt{2}}{2}X\right)-\nabla_{(Z+\sqrt2X)}Z=$$
$$\frac{1}{2}Y-\nabla_ZZ-\sqrt{2}\nabla_XZ=\frac{1}{2}Y-Y+Y=\frac{1}{2}Y,$$
$$k(Y,Z)=(R(Y,Z)Z,Y)=-\frac{1}{2}.$$

Obtained equalities imply all statements of Theorem \ref{main}.
\end{proof}

\begin{remark}
It is similar, that the left-invariant Lorentz metric on $\widetilde{\SL(2,\mathbb{R})}$ from \cite{HilgNeeb1993} is isometric to G\"odel metric $(S_0,ds^2_0)$ induced by G\"odel metric (\ref{lig}) with $a=\frac{1}{\sqrt{2}}.$
\end{remark}

\section{Isometry of non-isomorphic sub-Riemannian Lie groups}

Let us change notation $x_1\leftrightarrow x_2,$ $x_0\rightarrow x_3.$ Then the mapping
\begin{equation}
\label{red}	
\left(\begin{array}{cccc}
e^{-x_2} & 0 & 0 & x_1\\
0 & 1 & 0 & x_2\\
0 & 0 & 1 & x_3 \\
0 & 0 & 0 &  1 \end{array}\right)\rightarrow
\left(\begin{array}{ccc}
e^{-x_2} & 0 & x_1\\
0 & 1 &  x_3 \\
0 & 0 &  1 \end{array}\right) 
\in A^+(\mathbb{R})\times (\mathbb{R},+)
\end{equation}
is an isomorphism of matrix Lie groups with the basis of the Lie algebra
$$e_1=\left(\begin{array}{ccc}
0 & 0 & 1\\
0 & 0 & 0 \\
0 & 0 & 0 \end{array}\right),
e_2=\left(\begin{array}{ccc}
-1 & 0 & 0\\
0 & 0 &  0 \\
0 & 0 &  0 \end{array}\right),
e_3=\left(\begin{array}{ccc}
0 & 0 & 0\\
0 & 0 & 1 \\
0 & 0 & 0 \end{array}\right),$$
such that only $[e_1,e_2]=-[e_2,e_1]=e_1$ are unique nonzero Lie brackets.

Analogously to (\ref{isomgr}) and (\ref{isomgr1}), the mapping
\begin{equation}
\label{F}
F\left(\left(\begin{array}{ccc}
e^{-x_2} & 0 & x_1\\
0 & 1 &  x_3 \\
0 & 0 &  1 \end{array}\right)\right)=\left(\left(\begin{array}{cc}
e^{-x_2/2} & x_1\\
0 & e^{x_2/2} 
\end{array}\right), \left(\begin{array}{cc}
\cos x_3 & \sin x_3\\
-\sin x_3 & \cos x_3 
\end{array}\right)\right)
\end{equation}
is the universal covering epimorphism of $ A^+(\mathbb{R})\times (\mathbb{R},+)$ onto $\Sol(2)\times\SO(2).$

The standard left-invariant sub-Riemannian structure on 
$A^+(\mathbb{R})\times (\mathbb{R},+)$ is defined in \cite{AgrBar2012} by the orthonormal frame $\Delta=span\{e_2,e_1+e_3\}.$ 
Then there is unique sub-Riemannian structure
on $\Sol(2)\times \SO(2)$ such that $F$ is a local isometry; it is defined by the orthonormal frame $\overline{\Delta}=span\{\overline{e_2},\overline{e_1}+\overline{e_3}\},$
where
\begin{equation}
\label{base1}
\overline{e_1}=\left(\begin{array}{cc}
0 & 1\\
0 & 0 
\end{array}\right), \quad 
\overline{e_2}=\left(\begin{array}{cc}
-1/2 & 0\\
0 & 1/2 
\end{array}\right), \quad \overline{e_3}=\left(\begin{array}{cc}
0 & 1\\
-1 & 0 
\end{array}\right).
\end{equation}

Now we follow \cite{AgrBar2012}. Let $a=e^{-x_2}$ and $b=x_1$ in the second
matrix of (\ref{red}).

The subgroup $A^+(\mathbb{R})$ is diffeomorphic to the half-plane 
$\{(a,b)\in\mathbb{R}^2, a>0\},$ which is desrcibed in the standard polar
coordinates as $\{(\rho,\theta)| \rho>0, -\pi/2< \theta < \pi/2\}.$

\begin{theorem}
\label{AB} \cite{AgrBar2012}. The diffeomorphism $\Psi:A^+(\mathbb{R})\times \S^1\rightarrow \SL(2,\mathbb{R})$ defined by
\begin{equation}
\label{Psi}
\Psi(\rho,\theta,\varphi)=\frac{1}{\sqrt{\rho\cos\theta}}
\left(\begin{array}{cc}
\cos \varphi & \sin \varphi\\
\rho\sin (\theta-\varphi) & \rho\cos (\theta-\varphi) 
\end{array}\right),
\end{equation}
where $(\rho,\theta)\in A^+(\mathbb{R})$ and $\varphi\in \S^1,$ is a global
sub-Riemannian isometry.
\end{theorem}

\begin{remark}
\label{rem}
Using the above locally isometric covering $F$, we can and will understand
$\Psi$ as the global isometry between $\Sol(2)\times \SO(2)$ and $\SL(2,\mathbb{R})$ supplied with sub-Riemannian metrics defined by the same frame $\overline{\Delta}.$ 
\end{remark}

\begin{corollary}
$A^+(\mathbb{R})\times (\mathbb{R},+)$ with sub-Riemannian metric, defined by the frame $\Delta,$ is isometric to the universal covering $\widetilde{\SL(2,\mathbb{R})}$ of $\SL(2,\mathbb{R})$ with sub-Riemannian metric such that the 
natural universal covering epimorphism of $\widetilde{\SL(2,\mathbb{R})}$ onto $\SL(2,\mathbb{R})$ with sub-Riemannian metric, defined by the frame
$\overline{\Delta},$ is a local isometry.	
\end{corollary}

\begin{proposition}
\label{PsiAB}	
The global isometry $\Psi$ in the sense of Remark \ref{rem} is the
Iwasawa diffeomorphism of $\Sol(2)\times \SO(2)$ onto $\SL(2,\mathbb{R})$ of the view $(n\overline{a},k)\in NA\times \SO(2)\rightarrow n\overline{a}k\in NAK= \SL(2,\mathbb{R})$, where
$$n=\left(\begin{array}{cc}
1 & 0\\
b & 1 
\end{array}\right), \overline{a}=\left(\begin{array}{cc}
a^{-1/2} & 0\\
0 & a^{1/2} 
\end{array}\right), k=\left(\begin{array}{cc}
\cos\varphi & \sin\varphi\\
-\sin\varphi & \cos\varphi 
\end{array}\right), a=\rho\cos\theta, b=\rho\sin\theta.$$ 	
\end{proposition}

\begin{proof}
One needs simply to check that $n\overline{a}k$ is equal to the matrix in (\ref{Psi}).	
\end{proof}

\begin{remark}
Notice that $n=\exp(t\tilde{e}_1),$ $\overline{a}=\exp(s\overline{e_2}),$
where $\tilde{e}_1=(\overline{e}_1)^T,$ $^T$ is the sign of  transposition, $b=t$, and $a^{1/2}=e^{s/2}$. Also  $[\tilde{e}_1,\overline{e_2}]=-\tilde{e}_1.$
	
\end{remark}

\end{document}